\documentclass[letterpaper, 10pt]{article}

\usepackage{amssymb}
\usepackage{amsthm}
\usepackage{color}
\usepackage{graphicx}
\usepackage{ifthen}
\usepackage{mathrsfs} 
\usepackage{mathscinet} 
\usepackage{mathtools}
\usepackage{rotating}
\usepackage{stmaryrd}
\SetSymbolFont{stmry}{bold}{U}{stmry}{m}{n}   
\usepackage{linsys}
\usepackage{braket}
\usepackage{hyperref}
\usepackage{enumitem} 
\usepackage{mathtools}
\usepackage{pdfsync}

\newcommand*{\Z}{\mathbb{Z}}
\newcommand*{\Q}{\mathbb{Q}}
\newcommand*{\R}{\mathbb{R}}

\newcommand*{\verts}[1]{\left\lvert #1 \right\rvert}
\newcommand*{\braces}[1]{\left\lbrace #1 \right\rbrace}
\newcommand*{\parens}[1]{\left\lparen #1 \right\rparen}

\newcommand*{\card}{\verts}
\newcommand*{\setof}{\braces}

\newcommand*{\deftobe}{\mathrel{\coloneqq}}

\newcommand*{\maps}{\colon}
\newcommand*{\st}{\,:\,}

\renewcommand*{\epsilon}{\varepsilon}
\renewcommand*{\phi}{\varphi}
\renewcommand{\subset}{\subseteq}

\newtheorem{thm}{Theorem}[section]

\newtheorem{lemma}[thm]{Lemma}

\newtheorem{prop}[thm]{Proposition}

\theoremstyle{definition}

\theoremstyle{remark}

\newcommand{\ind}[1]{m_{#1}}

\renewcommand{\P}{\mathcal{P}}
\newcommand{\B}{\mathcal{B}}
\newcommand{\QQ}{\mathcal{Q}}
\newcommand{\RR}{\mathcal{R}}

\DeclareMathOperator{\conv}{Conv}

\newcommand*{\ehr}{\operatorname{ehr}}  
\newcommand*{\Ehr}{\operatorname{Ehr}}  

\newcommand{\pt}[1]{\mathbf{#1}}

\DeclareMathOperator{\Pyr}{\Delta}

\title{%
   Periods of Ehrhart coefficients of rational polytopes
}%

\author{Tyrrell B. McAllister\\
\small Department of Mathematics\\[-0.8ex]
\small University of Wyoming\\[-0.8ex] 
\small Laramie, WY, U.S.A.\\
\small\tt tmcallis@uwyo.edu\\
\and
H\'el\`ene O. Rochais%
\thanks{
    The second author was supported by a Wyoming EPSCoR Research
    Fellowship
}\\
\small Mathematics Department\\[-0.8ex]
\small California Institute of Technology\\[-0.8ex] 
\small Pasadena, CA, U.S.A.\\
\small\tt hrochais@caltech.edu\\
}

\date{January 31, 2018}

\begin{document}

\maketitle
   
\begin{abstract}
   Let $\P \subset \R^{n}$ be a polytope whose vertices have
   rational coordinates.  By a seminal result of E.~Ehrhart, the
   number of integer lattice points in the $k$th dilate of $\P$
   ($k$ a positive integer) is a quasi-polynomial function of~$k$
   --- that is, a ``polynomial'' in which the coefficients are
   themselves periodic functions of $k$.
   It is an open problem to determine which quasi-polynomials are
   the Ehrhart quasi-polynomials of rational polytopes.  As
   partial progress on this problem, we construct families of
   polytopes in which the periods of the coefficient functions
   take on various prescribed values.
\end{abstract}


\section{Introduction}

Let $\P \subset \R^{n}$ ($n \ge 2$) be a \emph{convex\footnote{We
will also consider nonconvex polytopes, as defined below.}
rational polytope} --- that is, the convex hull of finitely many
points in $\Q^{n}$.  By a famous theorem of Ehrhart
\cite{Ehr1962a}, the number of integer lattice points in positive
integer dilates $k\P$ of $\P$ is given by a quasi-polynomial
function of $k$.  In particular, there exist \emph{coefficient
functions} $c_{\P,i} \maps \Z \to \Q$ with finite periods such
that
\begin{equation}
   \label{eq:EhrhartQP}
   \card{k\P \cap \Z^{n}} = \sum_{i=0}^{n} c_{\P,i}(k) k^{i}, \qquad
   \text{for $k \in \Z_{\ge 1}$.}
\end{equation}
The function $\ehr_{\P} \maps \Z \to \Z$ defined by $\ehr_{\P}(x)
\deftobe \sum_{i=0}^{n} c_{\P,i}(x) x^{i}$ is the \emph{Ehrhart
quasi-polynomial} of $\P$.  (We refer the reader to
\cite{BecRob2007, Hib1992a, Sta1997} for introductions to Ehrhart
theory.)

The motivation for this paper is the problem of characterizing the
Ehrhart quasi-polynomials of rational polytopes.  It is well known
that, if $\P$ is an \emph{integral} polytope (meaning that all
vertices are in $\Z^{n}$), then the coefficients $c_{\P, i}$ are
constants and $\ehr_{\P}$ is simply a polynomial.  Already in this
case, the question of which polynomials are Ehrhart polynomials is
difficult.  Beginning with the pioneering work of Stanley
\cite{Sta1980, Sta1991}, Betke \& McMullen \cite{BetMcM1985}, and
Hibi \cite{Hib1994}, many inequalities have been shown to be
satisfied by the coefficients of Ehrhart polynomials of integral
polytopes.  (For recent work in this area, see \cite{BDLDPS05,
Bra2008, Bre2012, HNP2009, Hen2009, Pfe2010, Sta2016,
Sta2009} and references therein.)  Indeed, in the $2$-dimensional
case, a 1976 result of Scott \cite{Sco1976} completely
characterizes the Ehrhart polynomials of convex integral polygons.
Nonetheless, a complete characterization is not yet known for the
Ehrhart polynomials of convex integral polytopes in dimension $3$
or higher.

Much less is known about the characterization of Ehrhart
quasi-polynomials in the nonintegral case.  For example, even in
dimension $2$, we do not know which polynomials are the Ehrhart
polynomials of nonintegral convex polygons
\cite{McAMor2017}\footnote{Herrmann \cite{Her2010msthesis}
characterizes the Ehrhart quasi-polynomials of half-integral
not-neces\-sarily-convex polygons.}.

In this paper, we approach the problem of characterizing Ehrhart
quasi-poly\-no\-mials by focusing on the possible periods of the
coefficient functions $c_{\P,i}$ appearing in
equation~\eqref{eq:EhrhartQP}.  Define the \emph{period sequence}
of $\P$ to be $(p_{0}, p_{1}, \dotsc, p_{n})$, where $p_{i}$ is
the (minimum) period of $c_{\P,i}$.  Our question is thus:
\emph{What are the possible period sequences of rational
polytopes?}

If the interior of $\P$ is nonempty, then the leading coefficient
function $c_{\P,n}$ is a constant equal to the volume of $\P$, and
so $p_{n} = 1$.  More generally, McMullen \cite{McM1978} showed
that each coefficient period $p_{i}$ is bounded by the
corresponding \emph{$i$-index} of~$\P$.  The $i$-index is the
least positive integer $\ind{i}$ such that every $i$-dimensional
face of the dilate $\ind{i} \P$ contains an integer lattice point
in its affine span.  We call $(\ind{0}, \dotsc, \ind{n})$ the
\emph{index sequence} of $\P$.  McMullen proved the following.

\begin{thm}[McMullen \protect{\cite[Theorem 6]{McM1978}}]
\label{thm:McMullenBound}
   Let $\P$ be an $n$-dimensional rational polytope with period
   sequence $(p_{0}, \dotsc, p_{n})$ and index sequence $(\ind{0},
   \dotsc, \ind{n})$.  Then $p_{i}$ divides $\ind{i}$ for $0 \le i
   \le n$.  In particular, $p_{i} \le \ind{i}$.
\end{thm}
We will refer to the inequalities $p_{i} \le \ind{i}$ in Theorem
\ref{thm:McMullenBound} as \emph{McMullen's bounds}.  It is easy
to see that the indices $\ind{i}$ of a rational polytope satisfy
the divisibility relations $\ind{n} \mid \ind{n-1} \mid \dotsb
\mid \ind{0}$ and hence that $\ind{n} \le \ind{n-1} \le \dotsb \le
\ind{0}$.  Beck, Sam, and Woods \cite{BecSamWoo2008} showed that
McMullen's bounds are always tight in the $i = n-1$ and \mbox{$i =
n$} cases.  It is also shown in \cite{BecSamWoo2008} that, given
any positive integers \mbox{$m_{n} \mid m_{n-1} \mid \dotsb \mid
m_{0}$}, there exists a polytope with $i$-index $m_{i}$ for $0 \le
i \le n$.  Moreover, all of McMullen's bounds are tight for this
polytope.  This construction establishes the following.

\begin{thm}[\protect{Beck et~al.~\cite{BecSamWoo2008}}]
   \label{thm:BeckSamWoods}
   Let positive integers $p_{n-1} \mid p_{n-2} \mid \dotsb \mid
   p_{0}$ be given.  Then there exists an $n$-dimensional convex
   polytope in $\R^{n}$ with period sequence $(p_{0}, \dotsc,
   p_{n-1}, 1)$.
\end{thm}

The period sequences realized by Theorem \ref{thm:BeckSamWoods}
must satisfy the inequalities \mbox{$p_{0} \ge \dotsb \ge
p_{n-1}$}.  In the following sections, we extend the set of known
period sequences by constructing rational polytopes in which the
period sequences do not satisfy these inequalities.
Alternatively, our constructions may be thought of as examples in
which a particular one of McMullen's bounds is arbitrarily far
from tight.  Our first main result is the following, which is
proved in Section \ref{sec:1p11}.

\begin{thm}
   \label{thm:1p11}
   Let a positive integer $p$ be given.  Then there exists an
   $n$-dimensional convex rational polytope in $\R^{n}$ with
   period sequence $(1, p, 1, \dotsc, 1)$.
\end{thm}

In Theorem \ref{thm:1p11}, we achieve \emph{convex} polytopes.  In
other cases, we are unable to find convex constructions and must
consider nonconvex rational polytopes.  In general, we call a
topological ball $\B$ in $\R^{n}$ a \emph{rational polytope} (not
necessarily convex) if $\B$ is a union $\bigcup_{i \in I} \P_{i}$
of a finite family $\setof{\P_i \st i \in I}$ of convex rational
polytopes, all with the same affine span, in which every nonempty
intersection $\P_{i} \cap \P_{j}$, $i \ne j$, is a common facet of
$\P_{i}$ and~$\P_{j}$.

Our second main result, proved in Section \ref{sec:TheBarn}, is
the construction of \emph{non}convex polytopes with period
sequences of the form $(1, \dotsc, 1,p,1)$.
\begin{thm}
   \label{thm:TheBarn}
   Let a positive integer $p$ be given.  Then there exists an
   $n$-dimensional nonconvex polytope in $\R^{n}$ with period
   sequence $(1, \dotsc, 1, p, 1)$, provided that either \mbox{$3
   \le n \le 11$} or $n = 13$.
\end{thm}

\section{Building blocks}
\label{sec:BuildingBlocks}

In this section, we fix notation and recall results that will be
used in the constructions below.  We also establish Theorems
\ref{thm:1p11} and \ref{thm:TheBarn} in the dimension $n = 2$
case.

We are interested in the period sequences of quasi-polynomials.
This period sequence is invariant under addition of polynomials.
Thus, it will be convenient to consider quasi-polynomials $f(x)$
and $g(x)$ to be \emph{equivalent} when $f(x) - g(x)$ is a
polynomial.  In this case, we write $f(x) \equiv g(x)$.  In
particular, if $f(x) \equiv g(x)$, then $f(x)$ and $g(x)$ have the
same period sequence.  The chief convenience of this notation is
that, if $\QQ \cup \RR$ is a union of rational polytopes $\QQ$ and
$\RR$ such that $\QQ \cap \RR$ is integral, then $\ehr_{\QQ \cup
\RR} (x) \equiv \ehr_{\QQ}(x) + \ehr_{\RR}(x)$.

The constructions in the following sections depend on certain
$2$-dimensional polygons studied in \cite{McAMor2017}.  Let $p$ be
a positive integer and set $q \deftobe p^{2}-p+1$.  (Typically,
$p$ will be the desired period of a coefficient function in the
Ehrhart quasi-polynomial of a rational polytope.)  Let $\ell
\subset \R$ be the closed segment $[-\tfrac{1}{p}, 0]$.  Then the
Ehrhart quasi-polynomial of $\ell$ has the form
\begin{equation*}
   \ehr_{\ell}(x) = \tfrac{1}{p} x + c_{\ell, 0}(x),
\end{equation*}
where $c_{\ell, 0}$ has (minimum) period $p$.  Let $P \subset
\R^{2}$ be the convex pentagon with vertices $\pt{u}^{+}$,
$\pt{u}^{-}$, $\pt{v}^{+}$, $\pt{v}^{-}$, $\pt{w}$, where
\begin{align*}
   \pt{u}^{\pm} & \deftobe \pm q \pt{e}_{1}, &
   \pt{v}^{\pm} & \deftobe \pm (q-1) \pt{e}_{1} + \pt{e}_{2} , &
   \pt{w} & \deftobe \frac{q}{p}\pt{e}_{2}.
\end{align*}
(We write $\pt{e}_{i}$ for the $i$th standard basis vector.)  A
key fact, proved in \cite{McAMor2017}, is that the Ehrhart
quasi-polynomials of $P$ and $\ell$ are ``complements'' of each
other in the sense that the periodic parts of their coefficients
cancel when the quasi-polynomials are added together\footnote{In
\cite{McAMor2017}, the pentagon $P$ was reflected about the
diagonal $x = y$.}.  That is,
\begin{equation}
   \label{eq:PentagonEquivToSegment}%
   \ehr_{P}(x) \equiv -\ehr_{\ell}(x).
\end{equation}

As a warm-up for the following sections, we recall how $P$ and
$\ell$ were used in \cite{McAMor2017} to construct a polygon with
period sequence $(1,p,1)$.  Let $R\subset \R^{2}$ be the rectangle
$[-q, q] \times \ell$, and consider the convex heptagon $H
\deftobe \conv(R \cup P) \subset \R^{2}$.  Note that $R$ and $P$
are rational polygons whose intersection is the lattice segment
with endpoints $\pt{u}^{\pm}$.  Hence we can use equivalence
\eqref{eq:PentagonEquivToSegment} to compute that
\begin{align*}
   \ehr_{H}(x) & \equiv (2q x + 1) \ehr_{\ell}(x) + \ehr_{P}(x) \\
    & \equiv  (2q x + 1) \ehr_{\ell}(x) - \ehr_{\ell}(x) \\
    & = 2q x \ehr_{\ell}(x) \\
    & \equiv 2q \, c_{\ell, 0}(x) \, x.
\end{align*}
That is, $H$ has period sequence $(1, p, 1)$.  This establishes
the $n=2$ cases of both Theorem \ref{thm:1p11} and Theorem
\ref{thm:TheBarn}.

\section{Convex polytopes with period sequence \texorpdfstring{$(1, p, 1, \dotsc, 1)$}{(1, p, 1, ..., 1)}}
\label{sec:1p11}

Let a positive integer $p \ge 1$ be given.  Recall that we set $q
\deftobe p^{2} - p + 1$.  We now prove Theorem~\ref{thm:1p11} by
constructing a convex rational polytope $H_{n} \subset \R^{n}$
with period sequence $(1,p,1, \dotsc, 1)$.  Since the $n=2$ case
was established in the previous section, we assume that $n \ge 3$.

A useful fact about equivalence \eqref{eq:PentagonEquivToSegment}
is that it continues to hold when we take \mbox{$i$-fold} pyramids
over both $P$ and $\ell$.  More precisely, let $\QQ \subset
\R^{d}$ be a polytope, and let $\QQ'$ be the embedded copy of
$\QQ$ in $\R^{d+1}$ defined by $\QQ' \deftobe \setof{(\pt{x}, 0)
\in \R^{d+1} \st \pt{x} \in \QQ}$.  Fix a point $\pt{a} \in
\Z^{d+1}$ with final coordinate equal to $1$.  Then $\conv(\QQ'
\cup \setof{\pt{a}})$ is a ($1$-fold) \emph{pyramid} over $\QQ$.
By induction, for $i \ge 2$, define an \emph{$i$-fold pyramid}
over $\QQ$ to be a pyramid over an $(i-1)$-fold pyramid over
$\QQ$.

\begin{prop}
   \label{prop:PentagonEquivToSegment-Pyramids}
   Let $P$ and $\ell$ be the pentagon and line segment defined in
   the previous section, and let $\Pyr(P)$ and $\Pyr(\ell)$ be
   $i$-fold pyramids over $P$ and $\ell$, respectively.  Then
   \begin{equation}
   \label{eq:PentagonEquivToSegment-Pyramids}
      \ehr_{\Pyr(P)}(x) \equiv -\ehr_{\Pyr(\ell)}(x).
   \end{equation}
\end{prop}

\begin{proof}
   The \emph{Ehrhart series} $\Ehr_{\QQ}(t)$ of a rational
   polytope $\QQ$ is the generating function of $\ehr_{\QQ}(x)$.  
   That is, $\Ehr_{\QQ}(t)$ is the formal power series
   \begin{equation*}
      \Ehr_{\QQ}(t) \deftobe \sum_{k = 0}^{\infty} \ehr_{\QQ}(k)
      t^{k}.
   \end{equation*}
   It is well known that $\Ehr_{\QQ}(t)$ is a rational function in
   $t$.  Furthermore, if $\Pyr(\QQ)$ is an $i$-fold pyramid over
   $\QQ$, then
   \begin{equation*}
      \Ehr_{\Pyr(\QQ)}(t) = \frac{1}{(1 - t)^{i}} \Ehr_{\QQ}(t).
   \end{equation*}
   Given generating functions $F(t)$ and $G(t)$ of
   quasi-polynomials $f(x)$ and $g(x)$, respectively, we write
   $F(x) \equiv G(x)$ if $f(x) \equiv g(x)$.  Hence,
   equivalence~\eqref{eq:PentagonEquivToSegment} implies that
   $\Ehr_{P}(t) \equiv - \Ehr_{\ell}(t)$, and so
   \begin{align*}
      \Ehr_{\Pyr(P)}(t) = \frac{1}{(1-t)^{i}} \Ehr_{P}(t)
       \equiv -\frac{1}{(1-t)^{i}} \Ehr_{\ell}(t)
       = - \Ehr_{\Pyr(\ell)}(t).
   \end{align*}
   Equivalence \eqref{eq:PentagonEquivToSegment-Pyramids} follows
   by comparing coefficients of the series.
\end{proof}

One example of an $i$-fold pyramid that we will have occasion to
use is the simplex $S_{n} \subset \R^{n-1}$ given by
\begin{equation*}
   S_{n} \deftobe \conv\setof{\pt{0}, -\tfrac{1}{p}\pt{e}_{1},
   \pt{e}_{2}, \dotsc, \pt{e}_{n-1}},
\end{equation*}
which is an $(n-2)$-fold pyramid over $\ell$.  It is known that
the period sequence of $S_{n}$ is $(p, 1, \dotsc, 1)$.  Indeed, up
to a lattice-preserving transformation, $S_{n}$ is among the
polytopes constructed by Beck et al.~\cite{BecSamWoo2008} to prove
Theorem~\ref{thm:BeckSamWoods}.  

We also construct an $(n-2)$-fold pyramid over the pentagon $P$ as
follows.  Write $P' \subset \R^{n}$ for the embedded copy of $P$
defined by 
\begin{equation*}
    P' \deftobe \setof{(\pt{x}, 0, \dotsc, 0) \in \R^{n}
    \st \pt{x} \in P}.
\end{equation*}
We set $P_{n}$ to be the pyramid
\begin{equation*}
   P_{n} \deftobe \conv\parens{P' \cup \setof{\pt{e}_{3}, \dotsc, 
   \pt{e}_{n}}}.
\end{equation*}
Note that, by
Proposition~\ref{prop:PentagonEquivToSegment-Pyramids},
\begin{equation}
   \label{eq:SnEquivPn}
   \ehr_{S_{n}}(x) \equiv - \ehr_{P_{n}}(x).
\end{equation}

Let $W_{n} \subset \R^{n}$ be the translated prism over the
simplex $S_{n}$ defined by
\begin{equation*}
   W_{n} \deftobe ([-q, q] \times S_{n}) - q\pt{e}_{2},
\end{equation*}
where $[-q,q] \subset \R$ is a closed segment.  (The reason for
the translation by $-q\pt{e}_{2}$ is that it will make the convex
hull below easy to analyze.)  From the construction of $W_{n}$, it
follows that
\begin{equation*}
   \ehr_{W_{n}}(x) = (2qx + 1) \ehr_{S_{n}}(x).
\end{equation*}

We can now construct a convex polytope $H_{n} \subset \R^{n}$
which, we will show, has the period sequence $(1,p,1, \dotsc, 1)$.
Let
\begin{equation*}
   H_{n} \deftobe \conv \parens{ W_{n} \cup P_{n} }.
\end{equation*}
(See Figure \ref{fig:H3} for the case where $n = 3$ and $p = 2$
case.)  That $H_{n}$ has the desired period sequence is a direct
consequence of the following lemma.

\begin{figure}
   \centering \includegraphics[scale=0.40]{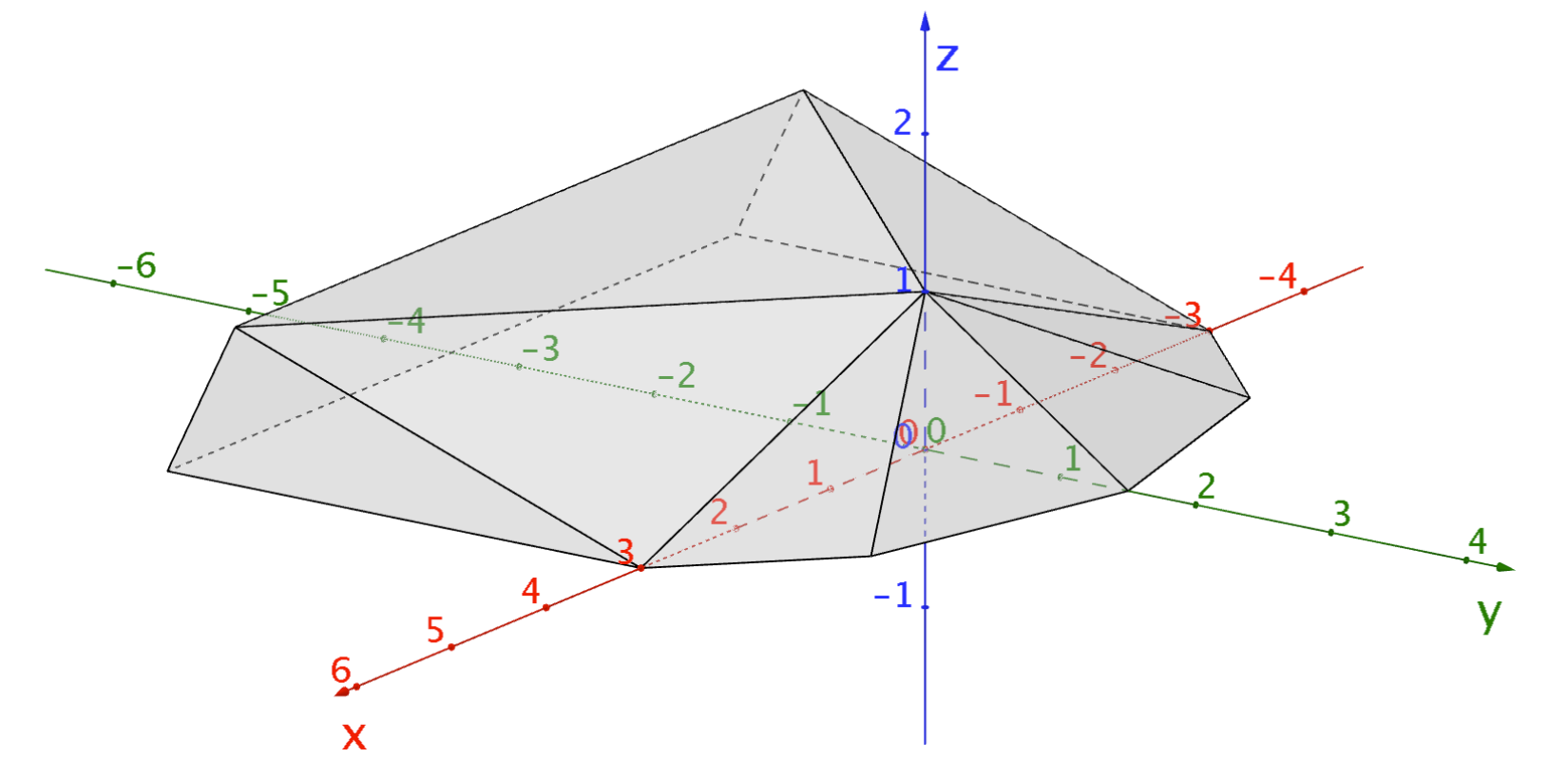}
   \caption{The polytope $H_{3}$ in the case where $p = 2$.}
   \label{fig:H3}
\end{figure}

\begin{lemma}
   The polytope $H_{n}$ defined above is a union of the form
   \begin{equation*}
      H_{n} = W_{n} \cup M_{n} \cup P_{n},
   \end{equation*}
   where $W_{n} \cap M_{n}$, $M_{n}$, and $M_{n} \cap P_{n}$ are
   lattice polytopes.
\end{lemma}
\begin{proof}
   Note that $W_{n}$, respectively $P_{n}$, has a facet
   perpendicular to the $\pt{e}_{2}$-axis.  Write $F_{W}$,
   respectively $F_{P}$, for this facet.  That is,
   \begin{align*}
      F_{W} & = \conv\setof{\pm q \pt{e}_{1}, \pm q \pt{e}_{1} +
      \pt{e}_{3}, \dotsc, \pm q \pt{e}_{1}+\pt{e}_{n}} - q 
      \pt{e}_{2}, \\
      F_{P} & = \conv\setof{\pm q \pt{e}_{1}, \pt{e}_{3}, \dotsc,
      \pt{e}_{n}}.
   \end{align*}
   Let $M_{n} \deftobe \conv (F_{W} \cup F_{P})$.  To prove that
   $H_{n} = W_{n} \cup M_{n} \cup P_{n}$, it suffices to prove the
   following two statements:
   \begin{enumerate}
      \item  
      For each facet $F \ne F_{W}$ of $W_{n}$, $P_{n}$ lies on the
      same side of the hyperplane supporting $F$ as $W_{n}$ does.
   
      \item  
      For each facet $F \ne F_{P}$ of $P_{n}$, $W_{n}$ lies on the
      same side of the hyperplane supporting $F$ as $P_{n}$ does.
   \end{enumerate}
   In other words, excepting $F_{W}$ and $F_{P}$, no facet of
   $W_{n}$ is visible from a vertex of $P_{n}$ and vice versa
   \cite[Section 22.3.1]{GooORo2004}.  
   
   To prove statement~(1), recall that $W_{n}$ is a prism over a
   simplex.  From this, the required facet-defining inequalities
   are easily determined and shown to be satisfied by the vertices
   of $P_{n}$.  To prove statement~(2), note that $P_{n}$ is a
   pyramid over $P_{n-1}$.  Hence, every facet of $P_{n}$ is
   either the ``base'' copy of $P_{n-1}$ or a pyramid over a facet
   of $P_{n-1}$ with apex $\pt{e}_{n}$.  Thus, the required
   facet-defining inequalities are again easily determined by
   induction and shown to be satisfied by the vertices of $W_{n}$.
\end{proof}

It is now straightforward to complete the proof of 
Theorem~\ref{thm:1p11}.  In particular, we prove the following:

\begin{thm}
   Let a positive integer $p$ be given.  Then the convex rational
   polytope $H_{n} \subset \R^{n}$ constructed above has period
   sequence $(1, p, 1, \dotsc, 1)$.
\end{thm}

\begin{proof}
   Apply equivalence \eqref{eq:SnEquivPn} to compute that
   $\ehr_{H_{n}}(x)$ satisfies
   \begin{equation*}
      \ehr_{H_{n}}(x)
       \equiv \ehr_{W_{n}}(x) + \ehr_{P_{n}}(x) 
       \equiv (2qx + 1) \ehr_{S_{n}}(x) - \ehr_{S_{n}}(x)
       = 2q x \ehr_{S_{n}}(x).
   \end{equation*}
   As noted at the beginning of this section, $S_{n}$ has period
   sequence $(p, 1, \dotsc, 1)$.  Therefore, $H_{n}$ has period
   sequence $(1, p, 1, \dotsc, 1)$, as desired.
\end{proof}

\section{Nonconvex polytopes with period sequence \texorpdfstring{$(1, \dotsc, 1, p, 1)$}{(1, ..., 1, p, 1)}}
\label{sec:TheBarn}

We again fix a positive integer $p \ge 1$.  The main result of
this section is the proof of Theorem~\ref{thm:TheBarn}.  In
particular, we construct $n$-dimensional \emph{non}convex rational
polytopes with period sequence $(1, \dotsc, 1,p,1)$, when \mbox{$3
\le n \le 11$} or $n = 13$.

The reason for the constraint on the dimension $n$ in
Theorem~\ref{thm:TheBarn} is that our construction depends upon
the existence of a solution to a particular system of Diophantine
equations in $n-1$ variables, namely, the so-called \emph{ideal
Prouhet--Tarry--Escott} (PTE) problem.  More precisely, we require
integers \mbox{$s_{1}, \dotsc, s_{n-1} > 0$} and $t_{1}, \dotsc,
t_{n-2} > t_{n-1} = 0$ such that
\begin{align}
   \label{eq:PTEconditions}
   \begin{split}
      p_{k}(s_{1}, \dotsc, s_{n-1})
       & = p_{k}(t_{1}, \dotsc, t_{n-1}) \qquad \text{for $0 \le k \le n-2$}, 
   \end{split}
\end{align}
where $p_{k}(\pt{x})$ is the power-sum symmetric function of
degree $k$ in $n-1$ variables.

Such solutions to system \eqref{eq:PTEconditions} are known to
exist when the number of variables is between $2$ and $10$
(inclusive) or is $12$ \cite[Chapter 11]{Bor2002}\footnote{An
integer solution to system \eqref{eq:PTEconditions} remains a
solution after a constant integer is added to all of the values
$s_{i}$, $t_{j}$.  Therefore, it suffices to find an integer
solution to \eqref{eq:PTEconditions} in which the minimum value
among the $s_{i}$, $t_{j}$ appears only once.  This condition is
satisfied by the solutions listed on \cite[p.~87]{Bor2002} to the
PTE problem.}.  No solution in $11$ variables is known.  Wright
\cite{Wri1934} conjectures that solutions exist for every number
of \mbox{variables $\ge 2$}.  However, Borwein
\cite[p.~87]{Bor2002} gives a heuristic argument suggesting that
this would be surprising.

Theorem \ref{thm:TheBarn} above is a corollary of the following
theorem, which constructs a nonconvex polytope with period
sequence $(1, \dotsc, 1, p, 1)$ in every dimension $n$ such that a
suitable solution to system \eqref{eq:PTEconditions} exists.

\begin{thm}
   Let a positive integer $p$ be given.  Let $n \ge 3$ be such
   that there are integers \mbox{$s_{1}, \dotsc, s_{n-1} > 0$} and
   $t_{1}, \dotsc, t_{n-2} > t_{n-1} = 0$ solving system
   \eqref{eq:PTEconditions} above.  Then there exists an
   $n$-dimensional polytope $B_{n} \subset \R^{n}$ with period
   sequence $(1, \dotsc, 1, p, 1)$.
\end{thm}

\begin{proof}
   Let $\ell = [-\tfrac{1}{p}, 0] \subset \R$, and let $P \subset \R^{2}$
   be the pentagon defined in Section~\ref{sec:BuildingBlocks}.
   Let $B_{n} \subset \R^{n}$ be the rational polytope defined as
   follows:
   \begin{equation*}
      B_{n} \quad \deftobe \quad
      \parens{\bigtimes_{i = 1}^{n-1}
      [0, s_{i}]} \times \ell
      \quad \cup \quad
      \bigg\lparen
         \bigtimes_{j = 1}^{n-2} [0, t_{j}]
      \bigg\rparen \times P.
   \end{equation*}
   Hence, $B_{n}$ is a union of two rational polytopes whose
   intersection is an integer polytope.  Thus,
   \begin{equation*}
      \ehr_{B_{n}}(x) \equiv 
      \parens{\prod_{i=1}^{n-1}(s_{i} x + 1)}  \ehr_{\ell}(x) 
      +
      \Bigg\lparen \prod_{j=1}^{n-2}(t_{j} x + 1) \Bigg\rparen \ehr_{P}(x).
   \end{equation*}
   Since $\ehr_{P}(x) \equiv -\ehr_{\ell}(x)$, it follows that
   \begin{equation*}
      \ehr_{B_{n}}(x) \equiv
      \Bigg\lparen \prod_{i=1}^{n-1}(s_{i} x + 1) -
      \prod_{j=1}^{n-2}(t_{j} x + 1) \Bigg\rparen \ehr_{\ell}(x).
   \end{equation*}
   We now exploit the fact that the $s_{i}$ and $t_{j}$ solve
   system \eqref{eq:PTEconditions}.  Newton's identities relating
   the power-sum symmetric functions to the elementary symmetric
   functions imply that the $s_{i}$ and $t_{j}$ also solve the
   system
   \begin{align*}
      \begin{split}
         e_{k}(s_{1}, \dotsc, s_{n-1})
          & = e_{k}(t_{1}, \dotsc, t_{n-1}) \qquad \text{for $0 \le k \le n-2$}, 
      \end{split}
   \end{align*}
   where $e_{k}(\pt{x})$ is the elementary symmetric function of
   degree $k$ in $n-1$ variables.  Hence,
   \begin{equation*}
      \prod_{i=1}^{n-1}(s_{i} x + 1) - \prod_{j=1}^{n-2}(t_{j} x +
      1) = s_{1}\dotsm s_{n-1} x^{n-1}.
   \end{equation*}
   Therefore, $\ehr_{B_{n}}(x) \equiv s_{1}\dotsm s_{n-1} x^{n-1}
   \ehr_{\ell}(x) \equiv s_{1}\dotsm s_{n-1} c_{\ell,0}(x)
   x^{n-1}$.  That is, all coefficient functions of
   $\ehr_{B_{n}}(x)$ are constant except for the coefficient of
   $x^{n-1}$, which has period $p$, as desired.
\end{proof}


\providecommand{\bysame}{\leavevmode\hbox to3em{\hrulefill}\thinspace}
\providecommand{\MR}{\relax\ifhmode\unskip\space\fi MR }
\providecommand{\MRhref}[2]{%
  \href{http://www.ams.org/mathscinet-getitem?mr=#1}{#2}
}
\providecommand{\href}[2]{#2}

\end{document}